\documentclass[a4paper
]{amsart}

\usepackage{mathrsfs}
\usepackage{todonotes}
\usepackage{enumitem}
\usepackage{array}
\usepackage{amssymb}
\usepackage{mathtools}
\usepackage{faktor}
\usepackage{xcolor}
\usepackage{tikz}
\usetikzlibrary{arrows}
\usetikzlibrary{decorations.pathmorphing}
\tikzset{snake it/.style={decorate, decoration=snake}}
\usepackage{tikz-cd}
\usepackage{amsthm}
\usepackage{thmtools}
\usepackage[breakable]{tcolorbox}
\usepackage[utf8]{inputenc}
\usepackage{csquotes}
\usepackage[english]{babel}
\usepackage{slashed}

\usepackage{newunicodechar}
\usepackage[style=alphabetic,backend=biber,maxnames=5,maxalphanames=5,doi=true,isbn=false,url=false]{biblatex}

\usepackage[hypertexnames=false]{hyperref} 
\hypersetup{urlcolor=blue, citecolor=blue, linkcolor=blue, colorlinks=true}
\usepackage[capitalise,nameinlink]{cleveref}
\crefname{enumi}{Theorem}{Main theorems}
\Crefname{enumi}{Theorem}{Main theorems}
\bibliography{literature.bib}
  
\renewbibmacro{in:}{}
\renewbibmacro*{issue+date}{%
  \ifboolexpr{not test {\iffieldundef{year}} or not test {\iffieldundef{issue}}}
    {\printtext[parens]{%
       \iffieldundef{issue}
         {\usebibmacro{date}}
         {\printfield{issue}%
          \setunit*{\addspace}%
          \usebibmacro{date}}}}
    {}%
  \newunit}

\newtheorem{theorem}{Theorem}

\newtheorem{lemma}[theorem]{Lemma}

\theoremstyle{definition}

\numberwithin{equation}{section}

\definecolor{grau1}{RGB}{100, 100, 100}
\definecolor{grau2}{RGB}{150, 150, 150}
\definecolor{grau3}{RGB}{200, 200, 200}

\newcommand{\oldcomment}[1]{
\begin{tcolorbox}[breakable]
	Old version in comments
\end{tcolorbox}
}

\newcommand{\proofcomment}[1]{
\begin{tcolorbox}[breakable]
	Proof in comments
\end{tcolorbox}
}

\DeclareMathAlphabet{\mathpzc}{OT1}{pzc}{m}{it}

\let\oldtocsection=\tocsection
\let\oldtocsubsection=\tocsubsection
\let\oldtocsubsubsection=\tocsubsubsection
\renewcommand{\tocsection}[2]{\hspace{0em}\oldtocsection{#1}{#2}}
\renewcommand{\tocsubsection}[2]{\hspace{1em}\oldtocsubsection{#1}{#2}}
\renewcommand{\tocsubsubsection}[2]{\hspace{2em}\oldtocsubsubsection{#1}{#2}}

\newcommand{\nocontentsline}[3]{}
\newcommand{\tocless}[2]{\bgroup\let\addcontentsline=\nocontentsline#1{#2}\egroup}

\setenumerate[1]{leftmargin=*,labelindent=0pt,label=(\roman*),}

\renewcommand{\phi}{\varphi}

\newcommand{\bbR}{\mathbb{R}}

\begin{document}

\author[Georg Frenck]{Georg Frenck}
\email{\href{mailto:georg.frenck@math.uni-augsburg.de}{georg.frenck@math.uni-augsburg.de}}
\urladdr{\href{http://frenck.net/Math}{frenck.net/math}\vspace{-.8em}}
\address{Universität Augsburg, Universitätsstr.~14, 86159 Augsburg, Germany}

\author[Bernhard Hanke]{Bernhard Hanke}
\email{\href{mailto:hanke@math.uni-augsburg.de}{hanke@math.uni-augsburg.de}}
\urladdr{\href{https://www.uni-augsburg.de/de/fakultaet/mntf/math/prof/diff/team/bernhard-hanke/}{math.uni-augsburg.de/hanke}\vspace{-.8em}}
\address{Universität Augsburg, Universitätsstr.~14, 86159 Augsburg, Germany}

\author[Sven Hirsch]{Sven Hirsch}
\email{\href{mailto:}{sven.hirsch@columbia.edu}}
\urladdr{\href{https://svenhirsch.com/}{svenhirsch.com}\vspace{-.8em}}
\address{Columbia University, 2990 Broadway, New York NY 10027, USA}

\title[The lock principle for scalar curvature]{The lock principle for scalar curvature}

\begin{abstract}
    We prove a Riemannian positive mass theorem for asymptotically flat spin manifolds with hypersurface singularities.
    Unlike earlier results, some components of the singular set may be mean-concave, provided that other components of the singular set are sufficiently mean-convex. 
    Our proof uses initial data sets where a suitably chosen second fundamental form transfers convexity defects between different singularity components. 
\end{abstract}

\maketitle

\section{Introduction}

Extremality and rigidity results in scalar curvature geometry on smooth Riemannian manifolds often remain valid for Riemannian metrics with singularities, provided that the relevant curvature restrictions are still satisfied in a weak sense. 
To give a well-known example, the positive mass theorem for asymptotically flat manifolds holds for manifolds with metric singularities along hypersurfaces, if the singularities are {\em mean-convex}, that is, the mean curvature jump with respect to the infinity pointing normal vector is non-positive, see \cite{shi-tam-manifolds-with-boundary, Miao2002cornersATP}.
As described in \cite[Proposition 3.1]{Miao2002cornersATP}, {\em mean-concave} singularities (with a strictly negative mean curvature jump) can be regarded as having infinitely negative scalar curvature in a weak sense.

\medskip

This paper shows that the above mean-convexity condition can be replaced with a weaker one: The effect of mean-concave jumps at some components of the singular set can be offset by mean-convex jumps at others. 
Loosely speaking, the {\em local} mean-convexity condition can be replaced by a weaker {\em global} one by passing to some average.
This phenomenon is unexpected, given that the pointwise lower scalar curvature bound in the positive mass theorem cannot be replaced by a lower bound on an averaged scalar curvature.

\medskip

To state our main result, let $M$ be a non-compact connected spin manifold which is decomposed along smooth compact hypersurfaces $\Sigma_1,\dots\Sigma_N$ as 
\[
M = M_0 \cup_{\Sigma_1} \dots\cup_{\Sigma_N} M_N.
\]
Here $M_0, \dots, M_N \subset M$ are smooth submanifolds with boundary, $M_{i}$ is compact for $1 \leq i \leq N-1$, and $M_N$ is noncompact and connected; see \cref{fig:corners}.
For $0 \leq i \leq N$, let $g_i$ be smooth Riemannian metrics on $M_i$ such that 
\begin{itemize}
    \item for $1 \leq i \leq N$, we have $(g_{i-1})|_{\Sigma_i} = (g_i)|_{\Sigma_i}$,
    \item each $g_i$ has non-negative scalar curvature, 
    \item $(M_{N},g_N)$ is asymptotically flat.
\end{itemize}
We write $g\coloneqq g_0\cup\dots\cup g_N$ for the induced $C^0$-Riemannian metric on $M$.
Let us denote the mean curvatures along $\Sigma_i$ by
\[H_{i,-}\coloneqq H(\Sigma_i\subseteq M_{i-1}), \qquad H_{i,+}\coloneqq H(\Sigma_i\subseteq M_{i}), \qquad 1 \leq i \leq N.
\]
We adopt the convention common in general relativity that mean curvatures are computed with respect to the infinity-pointing normal.
That is, if $M$ is the Euclidean space $\mathbb{R}^{n}$ with the decomposition
\[
    M_0 := \{ |x| \leq 1\}, \quad M_1:=\{1 \leq |x| \leq 2\}, \quad M_2 := \{|x| \geq 2\}, 
\]
then we have $H_{1,\mp} = n-1,\ H_{2,\mp } = (n-1)/2$.

\medskip

For $1 \leq i \leq N$, we set 
\begin{align*}
    \underline H_{i,-} := \min_{\Sigma_i} H_{i, -}, \qquad \overline H_{i,+} := \max_{\Sigma_i} H_{i, +}.
\end{align*}
Our main result is the following.

\begin{theorem}\label{thm main}
Assume that for $1 \leq i \leq N$, the mean curvatures $H_{i, \mp}$ are strictly positive, and that there exists some $1 \leq \Lambda < N$ such that 
\begin{align*}
\underline H_{i,-} \leq \overline H_{i,+} & \text{for } & 1 \leq i \leq \Lambda , \\
\underline H_{i,-} \geq \overline H_{i,+} & \text{for } & \Lambda + 1\leq i \leq N.
\end{align*}
Furthermore, assume that
\begin{align}\label{eq:condition-on-squares-of-mean-curvatures}
    \sum_{1 \leq i \leq N} \left(\underline H_{i,-}^2 - \overline H_{i,+}^2\right) \ge 0.
\end{align}
Then the ADM mass of $( M,g)$ is nonnegative.
\end{theorem}

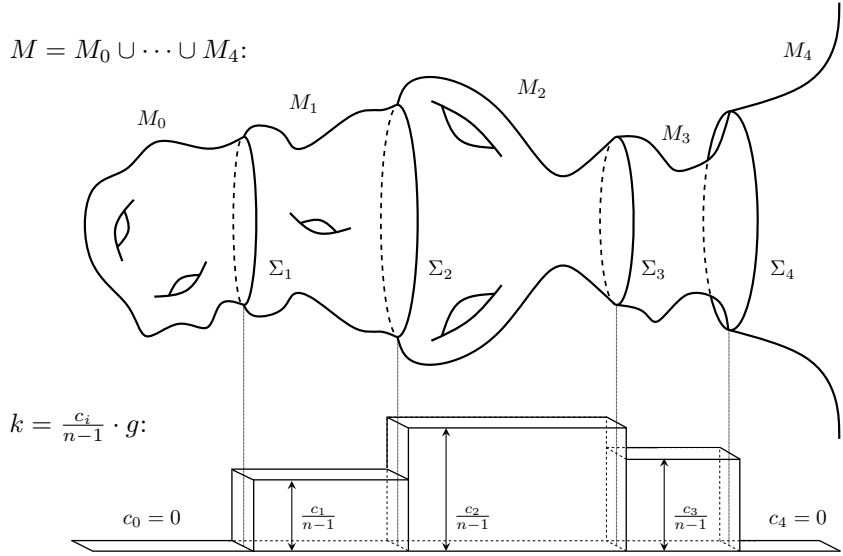
\begin{figure}[ht]
    \begin{tikzpicture}
        \node at (0,0) {\includegraphics[width=.8\textwidth]{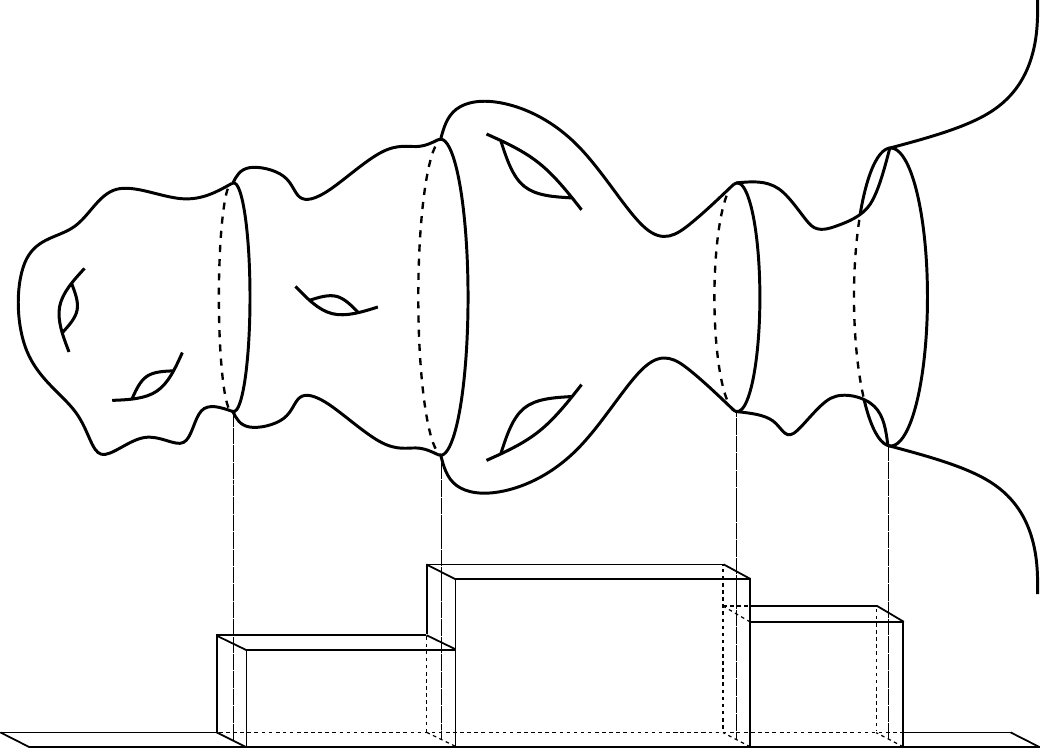}};
        \node[anchor=west] at (-6,3) {$ M = M_0\cup\dots\cup M_4$:};
        \node[scale=0.8] at (-4,2.1) {$M_0$};
        \node[scale=0.8] at (-2,2.3) {$M_1$};
        \node[scale=0.8] at (1,2.5) {$M_2$};
        \node[scale=0.8] at (2.9,1.9) {$M_3$};
        \node[scale=0.8] at (4.5,3) {$M_4$};
        \node[scale=0.8] at (-2.3,0.1) {$\Sigma_1$};
        \node[scale=0.8] at (-0.2,0.1) {$\Sigma_2$};
        \node[scale=0.8] at (2.6,0.1) {$\Sigma_3$};
        \node[scale=0.8] at (4.3,0.1) {$\Sigma_4$};
        
        \draw[stealth-stealth](-2.16,-3.63) to (-2.16,-2.7);
        \draw[stealth-stealth](-.13,-3.63) to (-.13,-2);
        \draw[stealth-stealth](2.74,-3.63) to (2.74,-2.4);

        \node[anchor=west] at (-6,-2) {$k=\frac{c_i}{n-1}\cdot g$:};
        \node[scale=0.8] at (-4,-3.2) {$c_0=0$};
        \node[scale=0.8] at (-1.81,-3.2) {$\frac{c_1}{n-1}$};
        \node[scale=0.8] at (0.22,-3.2) {$\frac{c_2}{n-1}$};
        \node[scale=0.8]at (3.09,-3.2) {$\frac{c_3}{n-1}$};
        \node[scale=0.8] at (4.5,-3.2) {$c_4=0$};
    \end{tikzpicture}
    \caption{The lock principle for $N=4$ and $\Lambda=2$.
    Two mean-concave curvature jumps at $\Sigma_1$ and $\Sigma_2$ are offset by two mean-convex curvature jumps at $\Sigma_3$ and $\Sigma_4$.
    The constants $c_i$ determining the symmetric tensor $k$ are defined in \eqref{eq:def-of-c-and-d}.}\label{fig:corners}
\end{figure}

For the proof we view $( M,g)$ as a time-symmetric slice of an initial data set $(M^n,g,k)$ with symmetric $(0,2)$-tensor field $k$.
When chosen appropriately, the tensor $k$ acts as a \emph{lock}: it allows one to convert an unfavorable mean curvature jump at $\Sigma_i$ into a corresponding jump of $k$, to transport this defect across $ M_i$, and finally to absorb it at a later singularity $\Sigma_{j}$ with a favorable mean curvature jump.
A typical situation is shown in \cref{fig:corners}.
In this situation, we can then apply the positive mass theorem for initial data sets with creases proven by Kazaras--Khuri--Lin \cite{KazarasKhuriLin2025}.
By moving to initial data sets, we thus prove a result in Riemannian geometry that cannot be achieved by local deformation arguments.

\medskip

This paper elaborates on a particular example of the lock principle. 
In \cref{sec extensions}, we discuss possible extensions.

\medskip

\noindent \textbf{Acknowledgements:} BH and SH are grateful to the Lonavala Geometry Festival for the ideal working conditions where part of this work was carried out. 
BH was partially supported by the DFG-SPP 2026 ``Geometry at Infinity''.

\section{Proof\texorpdfstring{ \cref{thm main}}{ Theorem \ref*{thm main}}}

Our main result follows from iterative application of the following lemma.

\begin{lemma}\label{lemma main}
    Suppose that $\Sigma$ is a smooth $(n-1)$-manifold and that $H_-,H_+\colon \Sigma\to\bbR$ are smooth functions, and let $\underline H_{-},\overline H_+\in(0,\infty)$ be such that $H_-\ge \underline H_-$ and $H_+\le\overline H_+$.
    Given $a\in\bbR$ satisfying $\overline{H}_+^2-\underline{H}_-^2+a^2\ge0$, we define functions $\vec{H}_+, \vec{H}_- \colon \Sigma \to \mathbb{R}^{1,1}$ via
    \begin{align*}
       \vec H_-=\begin{pmatrix}
            H_-\\
            a
        \end{pmatrix},
        \qquad
        \vec H_+=\begin{pmatrix}
            H_+\\
            \sqrt{\overline H_+^2-\underline H_-^2+a^2}
        \end{pmatrix}.
    \end{align*}
    Then there exists a hyperbolic rotation angle $\vartheta\in \mathbb R$ such that the function
    \begin{align*}
        X=F_\vartheta \vec H_--\vec H_+ \colon \Sigma \to \mathbb R^{1,1}, \qquad F_\vartheta=
        \begin{pmatrix}
            \cosh \vartheta & -\sinh \vartheta\\
            -\sinh \vartheta &\cosh\vartheta
        \end{pmatrix},
    \end{align*}
  satisfies
    \begin{align*}
        X_1\ge|X_2| \quad \textrm{ along } \Sigma,
    \end{align*}
    where $X_1$ and $X_2$ denotes the first and second component of $X$, respectively.
\end{lemma}

\begin{proof}
Let us abbreviate $\xi\coloneqq\sqrt{\overline H_+^{2}-\underline H_-^{2}+a^2}$.
We choose 
\begin{align}\label{eq vartheta 1}
    \vartheta =\sinh^{-1}\left( 
    \frac{a\overline H_+ - \xi\underline H_-}
    {\underline H_-^{2}-a^2}
    \right),
\end{align}
in case $\underline H_-\ne a$, and
\begin{align}\label{eq vartheta 2}
    \vartheta=\log(\underline H_-)-\log(\overline H_+)
\end{align}
in case $\underline H_-=a$.
Using $\cosh(\vartheta) = \sqrt{1+\sinh^2(\vartheta)}$, we obtain
\begin{align*}
   & F_\vartheta \vec H_- - \vec H_+
    =
    \begin{pmatrix}
        \sqrt{1+\sinh^2(\vartheta)}\,H_- - \sinh(\vartheta)\, a - H_+\\[0.4em]
        \sqrt{1+\sinh^2(\vartheta)}\,a - \sinh(\vartheta)\, H_- 
        - \xi
    \end{pmatrix}.
\end{align*}
If $\underline H_-\neq a$, then
\begin{align*}
    &\left(\underline H_-^{2}-a^2\right)^2\left(1+\sinh^2(\vartheta)\right)\\
    &\qquad\qquad = \underline H_-^4 - 2a^2\underline H_-^2 + a^4 + (a\overline H_+ - \xi\underline H_-)^2\\
    &\qquad\qquad = \underline H_-^4 - 2a^2\underline H_-^2 + a^4 + a^2\overline H_+^2 - 2a\xi\overline H_+\underline H_- + (\overline H_+^2-\underline{H}_-^2+a^2)\underline H_-^2\\
    &\qquad\qquad = \left(\overline H_+ \underline H_- - a\xi\right)^2
\end{align*}
and hence $\left|\underline H_-^{2}-a^2\right|\sqrt{1+\sinh^2(\vartheta)}=\left|\overline H_+\underline H_- - a \xi\right|$.

\medskip

In the case that $\underline H_->a$, we note that $\xi< \overline H_+$, and that both $\underline H_-^{2} -a^2$ and  $\overline H_+ \underline{H}_--a\xi$ are positive. 
We can thus estimate
\begin{align*}
    &(\underline H_-^{2} - a^2) 
    \left(\sqrt{1+\sinh^2(\vartheta)}\,H_- - \sinh(\vartheta)\, a - H_+\right)\\
        &\qquad\qquad= H_- \left(\overline H_+ \underline H_- - a \xi \right) - a \left(
        a\overline H_+ - \xi\underline H_-\right)- H_+(\underline H_-^{2}-a^2)\\
        &\qquad\qquad\ge(H_- -\underline H_-)\left(\overline H_+ \underline H_- - a \xi\right),
\end{align*}
where we used $H_+(\underline H_-^{2}-a^2)\le \overline H_+(\underline H_-^{2}-a^2)$ for the final inequality.
Moreover, we have
\begin{align*}
    &(\underline H_-^{2}-a^2)\left(\sqrt{1+\sinh^2(\vartheta)}\,a - \sinh(\vartheta)\, H_- - \xi\right)\\
        &\qquad\qquad= a \left(\overline H_+ \underline H_- - a \xi\right) - H_- \left(a\overline H_+ - \xi\underline H_-\right)- \xi(\underline H_-^{2}-a^2)\\
        &\qquad\qquad=(H_--\underline H_-)\left(\xi\underline H_--a\overline H_+\right)
\end{align*}
Since
\begin{align}\label{computation1}
\begin{split}
    &\left(\overline H_+ \underline H_- - a \xi \right)^2 -\left(\xi\underline H_- -a\overline H_+\right)^2\\
        &\qquad=\overline H_+^2\underline H_-^2- 2a\xi\overline H_+\underline H_- + a^2\xi^2 - \underline H_-^2\xi^2 + 2a\xi\overline H_+\underline H_- - a^2\overline H_+^2\\
        &\qquad= \overline H_+^2\underline H_-^2 + a^2 \left(\overline H_+^{2}-\underline H_-^{2}+a^2\right) - \underline H_-^2\left(\overline H_+^{2}-\underline H_-^{2}+a^2\right)- a^2 \overline H_+^2\\
        &\qquad =a^4-2a^2\underline H_-^2+\underline H_-^4 \ge0,
    \end{split}
\end{align}
we obtain $\overline H_+ \underline H_- - a\xi\ge |\xi\underline H_--a\overline H_+|$ and consequently,
\begin{align}\label{computation2}
    \begin{split}
        &\left(\sqrt{1+\sinh^2(\vartheta)}\,H_- - \sinh(\vartheta)\, a - H_+\right)\\
            &\qquad\qquad\qquad\ge \Big|\sqrt{1+\sinh^2(\vartheta)}\,a - \sinh(\vartheta)\, H_- - \xi\Big|.
    \end{split}
\end{align}
In the case $\underline H_-<a$, where $\xi> \overline H_+$ and both $\underline H_-^{2} - a^2$ and $\overline H_+ \underline{H}_--a\xi$ are negative, we can similarly estimate
\begin{align*}
    &-(\underline H_-^{2} - a^2) 
    \left(\sqrt{1+\sinh^2(\vartheta)}\,H_- - \sinh(\vartheta)\, a - H_+\right)\\
        &\qquad\qquad= -H_- \left(\overline H_+ \underline H_- - a \xi \right) + a \left(
        a\overline H_+ - \xi\underline H_-\right)+ H_+(\underline H_-^{2}-a^2)\\
        &\qquad\qquad\ge (H_- -\underline H_-)\left(-\overline H_+ \underline H_- + a \xi \right)\\
        &\qquad\qquad\ge (H_- -\underline H_-)\left|a \overline H_+ - \xi\underline H_- \right|,
\end{align*}
where the last inequality follows as in \eqref{computation1}.
Moreover,
\begin{align*}
    &-(\underline H_-^{2}-a^2)\left(\sqrt{1+\sinh^2(\vartheta)}\,a - \sinh(\vartheta)\, H_- - \xi\right)\\
        &\qquad\qquad= -a \left(\overline H_+ \underline H_- - a \xi\right) + H_- \left(a\overline H_+ - \xi\underline H_-\right)+ \xi(\underline H_-^{2}-a^2)\\
        &\qquad\qquad=(H_--\underline H_-)\left(a\overline H_+-\xi\underline H_-\right)
\end{align*}
Consequently, \eqref{computation2} holds in this case as well.

\medskip

Finally, we have in case $a=\underline H_-$. In this case $\xi=\overline H_+$ and therefore
\begin{align*}
    F_\vartheta \vec H_- - \vec H_+ = \begin{pmatrix}
        \sqrt{1+\sinh^2(\vartheta)}\,H_- - \sinh(\vartheta)\, \underline H_- - H_+\\[0.4em]
        \sqrt{1+\sinh^2(\vartheta)}\,\underline H_- - \sinh(\vartheta)\, H_- 
        - \overline H_+
    \end{pmatrix}.
\end{align*}
Note that $\sinh(\vartheta)=\tfrac12(\underline H_-\overline H_+^{-1}-\overline H_+\underline H_-^{-1})$ and 
\[\cosh(\vartheta)=\sqrt{1+\sinh^2(\vartheta)} = \tfrac{1}{2}(\underline H_-\overline H_+^{-1}+\overline H_+\underline H_-^{-1}).\]
Using $\overline H_+\ge H_+$, we obtain
\begin{align*}
    &\sqrt{1+\sinh^2(\vartheta)}\,H_- - \sinh(\vartheta)\, \underline H_- - H_+\\
        &\qquad\qquad\ge\frac12\left( H_- - \underline{H}_-\right) \left(\underline H_-\overline H_+^{-1}+ \overline H_+ \underline H_-^{-1}\right) = \left( H_- - \underline{H}_-\right) \cosh(\vartheta),\\
    &\sqrt{1+\sinh^2(\vartheta)}\,\underline H_- - \sinh(\vartheta)\, H_- - \overline H_+\\
        &\qquad\qquad=\frac12\left( H_- - \underline{H}_-\right) \left(\overline H_+ \underline H_-^{-1} - \underline H_-\overline H_+^{-1}\right) = -\left( H_- - \underline{H}_-\right) \sinh(\vartheta).
\end{align*}
As $\left( H_- - \underline{H}_-\right)\ge0$ and $\cosh(\vartheta)\ge |\sin(\vartheta)|$, we again obtain \eqref{computation2}, which finishes the proof.
\end{proof}

\begin{proof}[Proof of \cref{thm main}]
    Put $d_0=c_0= 0$, and for $1 \leq \ell \leq N$, put 
    \begin{align}\label{eq:def-of-c-and-d}
        d_\ell :=\sum_{i=1}^\ell(\overline H_{i,+}^2-\underline H_{i,-}^2), \quad c_{\ell} :=  \sqrt{ \max \{d_\ell, 0\}}. 
    \end{align}
    Without loss of generality, we can assume that $d_{\ell} > 0$ for some $\ell$. 
    Otherwise, \cref{thm main} is implied by \cite{Miao2002cornersATP}.

    \medskip

    Let $\Lambda' \in \{\Lambda , \ldots, N\}$ maximal with $d_{\Lambda'} > 0$.
    By assumption, we have $\Lambda' < N$.
    Furthermore, for $1 \leq \ell \leq \Lambda'$, we have $c_{\ell} = \sqrt{d_\ell}$, and for $\Lambda' +1 \leq \ell \leq N$, we have $c_\ell = 0$.
    For $0 \leq \ell \leq N$, define a symmetric $(0,2)$-tensor $k$ on $ M_{\ell}$ by setting 
    \begin{align*}
        k_\ell=\frac{c_\ell}{n-1} \cdot g_{\ell}.
    \end{align*}
    In particular, we have $k_N = 0$. 
    Moreover, since each $g_i$ has non-negative scalar curvature, the dominant energy condition holds on $M_0,\dots, M_N$.

    \medskip

    Let $\ell \in \{ 1, \ldots, \Lambda' \}$.
    If we choose 
    \begin{align}\label{eq:choice-for-lemma}
        \Sigma\coloneqq\Sigma_\ell,\quad H_\mp := H_{\ell,\mp},\quad \underline H_-\coloneqq \underline H_{\ell,-},\quad \overline H_+\coloneqq \overline H_{\ell,+}\quad\text{and}\quad a:=c_{\ell-1}
    \end{align}
    then $c_\ell = \sqrt{d_\ell} = \sqrt{\overline H_+^2-\underline H_-^2+a^2}$, and we obtain from \cref{lemma main} a hyperbolic angle $\vartheta_\ell$ such that
    \begin{align*}
        X=F_{\vartheta_\ell}\begin{pmatrix}
            H_{\ell,-}\\ c_{\ell-1}
        \end{pmatrix}
        -\begin{pmatrix}
            H_{\ell,+}\\ c_\ell
        \end{pmatrix}\in C^\infty\left(\Sigma_\ell,\mathbb{R}^{1,1}\right)
    \end{align*}
    satisfies $X_1\ge|X_2|$.

    \medskip

    For $\ell \in \{ \Lambda' +1, \ldots, N \}$, we choose $\Sigma$, $H_-$, $H_+$, $\underline H_-$, and the constant $a$ as in \eqref{eq:choice-for-lemma}, and we choose 
    \[\overline H_+ := \sqrt{ \underline H_{\ell,-}^2-c^2_{\ell-1}}.\]
    Note, that by the assumption on $\Lambda'$, we have $\underline H_{\ell,-}^2-c^2_{\ell-1}\ge \overline H_{\ell,+}^2$ and hence $\overline H_+ \geq \overline H_{\ell,+}\ge H_+$ and $\overline H_+^2 - \underline H_-^2 + a^2 = 0$.
    Therefore, an application of \cref{lemma main} gives a $\vartheta\in\bbR$ such that
    \begin{align*}
        X=F_{\vartheta_\ell}\begin{pmatrix}
            H_{\ell,-}\\ c_{\ell-1}
        \end{pmatrix}
        -\begin{pmatrix}
            H_{\ell,+} \\ 0 
        \end{pmatrix}\in C^\infty\left(\Sigma_\ell,\mathbb{R}^{1,1}\right)
    \end{align*}
    again satisfies $X_1\ge|X_2|$.

    \medskip

    Therefore, \cite[(1.7)]{KazarasKhuriLin2025} is satisfied along $\Sigma_\ell$, $\ell=1,\dots, N$,\footnote{In the notation of \cite{KazarasKhuriLin2025}, we have $\nu_+ = \binom{1}{0}$, $\tau_+ = \binom{0}{1}$ and $\beta_\pm=k(\nu,\cdot)|_{TM}=0$, $\beta^\Delta=0$. So, in our setting, \cite[(1.7)]{KazarasKhuriLin2025} transforms to the condition that $X=F_{\vartheta_\ell}\vec H_{\ell,-} -\vec H_{\ell,+}$ satisfies $X_1\ge|X_2|$.} that is the dominant energy condition holds weakly across each $\Sigma_\ell$.
    Consequently, the result follows from \cite[Theorem 1.3]{KazarasKhuriLin2025}, thus finishing the proof of \cref{thm main}.
\end{proof}

\section{Concluding remarks} \label{sec extensions}
We expect that several extensions and generalizations of the lock-principle are possible. In this section, we list a few of those without going into details.
\medskip
\begin{enumerate}[itemsep = .5em]
    \item Instead of the final jump, the unfavorable $k$-terms could also be absorbed at infinity by an asymptotically hyperboloidal end. 
    This mechanism plays a central role in the resolution of Gromov’s total mean curvature conjecture in \cite[Section 3] {FrenckHankeHirsch2026}.

    \item The assumption that the mean curvatures $ H_{i,\pm}$ are positive is not necessary. Allowing variable signs leads to a slightly stronger assumption in \eqref{eq:condition-on-squares-of-mean-curvatures} but preserves the underlying principle; see again \cite{FrenckHankeHirsch2026}.
    In the same spirit, the condition on the maximum and minimum of $H_{i,\pm}$ can be replaced by a pointwise condition. 

    \item Assumption \eqref{eq:condition-on-squares-of-mean-curvatures} could be weakened by incorporating the distance between corners, as in recent work on the positive mass theorem with shields and arbitrary ends \cite{LesourdUngerYau2024, CecchiniZeidler2024}.
    In other words, the farther away the jumps are from each other, the smaller the good jump has to be to compensate the bad jump.

    \item By an appropriate choice of $k$, any hypersurface $\Sigma_i$ can be transformed into a marginally outer trapped surface or a trapped surface which is an admissible boundary conditions for the spacetime positive mass theorem. 
    This opens the possibility of combining the lock principle with Penrose-type inequalities.
\end{enumerate}

\printbibliography
\vspace{1em}

\end{document}